\documentclass{article}

\usepackage{amsmath}
\usepackage{amssymb}
\usepackage{graphicx}
\usepackage{bm}
\usepackage{mathdots}
\usepackage{booktabs}
\usepackage{rotating}
\usepackage{listings}
\usepackage{url}
\usepackage[ruled,linesnumbered]{algorithm2e}
\usepackage[a4paper,left=2.8cm,right=2.8cm,top=2.5cm,bottom=2.5cm]{geometry}
\usepackage{fancyhdr}
\usepackage[colorlinks,
            linkcolor=black,
            anchorcolor=blue,
            citecolor=black,
            urlcolor=blue           ]{hyperref}

\usepackage[framemethod=tikz]{mdframed}
\newmdtheoremenv[%
backgroundcolor=green!10,%
outerlinecolor=black,%
leftmargin=0,%
rightmargin=0,
innertopmargin =3pt,%
innerleftmargin = 5pt,
innerrightmargin = 5pt,
splittopskip = \topskip,%
skipabove = \baselineskip,%
skipbelow = \baselineskip,%
roundcorner=5, ntheorem]
{theorem}{Theorem}[section]
\newtheorem{corollary}{Corollary}[section]

\newtheorem{lemma}{Lemma}[section]

\makeatletter 
\@addtoreset{equation}{section}
\makeatother  

\newtheorem{remark}{Remark}[section]

\newenvironment{proof}{{\noindent\it Proof.}\quad}{\hfill $\square$\\}

\begin{document}
\title{Marcinkiewicz--Zygmund inequalities for scattered data on polygons}
\author{Hao-Ning Wu\footnotemark[2]}

\renewcommand{\thefootnote}{\fnsymbol{footnote}}
\footnotetext[2]{Department of Mathematics, University of Georgia, Athens, GA 30602, USA (\href{hnwu@uga.edu}{hnwu@uga.edu}, \href{https://haoningwu.github.io/}{haoningwu.github.io)}}

\maketitle

\begin{abstract}

Given a set of scattered points on a regular or irregular 2D polygon, we aim to employ them as quadrature points to construct a quadrature rule that establishes Marcinkiewicz--Zygmund inequalities on this polygon. The quadrature construction is aided by Bernstein--B\'{e}zier polynomials. For this purpose, we first propose a quadrature rule on triangles with an arbitrary degree of exactness and establish Marcinkiewicz--Zygmund estimates for 3-, 10-, and 21-point quadrature rules on triangles. Based on the 3-point quadrature rule on triangles, we then propose the desired quadrature rule on the polygon that satisfies Marcinkiewicz--Zygmund inequalities for $1\leq p \leq \infty$. As a byproduct, we provide error analysis for both quadrature rules on triangles and polygons. Numerical results further validate our construction.
\end{abstract}

\textbf{Keywords: }{Marcinkiewicz--Zygmund inequalities, Bernstein--B\'{e}zier polynomials, quadrature, interpolation, scattered points, triangles, polygons}

\textbf{AMS subject classifications.} 41A17, 41A05, 65D32, 42C15

\section{Introduction}

Let $N\geq 1$ be an integer and $\chi$ a trigonometric polynomial of order at most $N$. The classical Marcinkiewicz--Zygmund inequality (see, e.g., \cite[Chapter X]{zygmund1977trigonometric}) states that 
\begin{equation*}
\int_0^{2 \pi}|\chi(x)|^p dx \leq \frac{A_p}{2 N+1} \sum_{k=0}^{2 N}\left|\chi\left(\frac{2 k \pi}{2 N+1}\right)\right|^p \leq A A_p \int_0^{2 \pi}|\chi(x)|^p d x
\end{equation*}
for $1<p<\infty$, where $A>0$ is a generic constant and $A_p>0$ is a constant depending only on $p$. The second inequality holds for $p=1$ and $p=\infty$ as well. Moreover, the first inequality is also valid for $p=1$ and $p=\infty$ if the summation includes more points than $2N+1$. These types of inequalities are of great practical value, as they ease the restrictive assumption of quadrature exactness in traditional numerical analysis. For instance, Marcinkiewicz--Zygmund inequalities have been applied to hyperinterpolation, originally proposed by Sloan in \cite{sloan1995polynomial} as a discretization of the $L^2$ orthogonal projection, to bypass the quadrature exactness assumption in its construction, see, e.g., \cite{an2022quadrature,an2024bypassing,an2024hyperinterpolation}. Recently, they have also been used to relax quadrature exactness requirements in classical numerical methods for partial differential equations \cite{wu2023breaking} and integral equations \cite{an2024spherical}.

Since the pioneering work \cite{Marcinkiewicz1937} of Marcinkiewicz and Zygmund, there has been extensive research on Marcinkiewicz--Zygmund inequalities across various domains. In one dimension, these inequalities have been studied on the torus (see, e.g., \cite{zbMATH00539969}), real lines (see, e.g., \cite{zbMATH01150002,zbMATH01398363}), and unit circles (see, e.g., \cite{zbMATH01289969}). It was followed by significant advancements in understanding Marcinkiewicz--Zygmund inequalities on compact manifolds, including spheres (see, e.g., \cite{zbMATH05558973,zbMATH07844329,zbMATH05206088,mhaskar2001spherical}) and general compact manifolds (see, e.g., \cite{zbMATH05795114,filbir2011marcinkiewicz,zbMATH07227728,zbMATH05271081}). Compared to the extensive literature on compact manifolds, the exploration of Marcinkiewicz--Zygmund inequalities on multivariate compact domains in Euclidean spaces remains relatively limited. Notable studies include results on spherical caps \cite{zbMATH05676005} and initial findings in several multivariate domains \cite{MR3746524}.

In this paper, we are driven by practical considerations on 2D polygons. Let $\Omega \subset \mathbb{R}^2$ be a regular or irregular 2D polygon, with its boundary $\partial \Omega$ consisting of piecewise line segments. Given a set of scattered points ${(x_j, y_j)}_{i=1}^m$ on $\Omega$ with vertices of $\Omega$ included, we are motivated to construct a quadrature rule using these scattered points as quadrature points and investigate the related Marcinkiewicz--Zygmund inequalities. This construction immediately links the extensive theory of Marcinkiewicz--Zygmund inequalities to practice. An analogue of our work is \cite{zbMATH07844329}, in which Marcinkiewicz--Zygmund inequalities for scattered data on spheres have been established.

Our approach involves triangulating $\Omega$ into triangles with these scattered points as vertices. We first focus on a single triangle, proposing a quadrature rule in its domain points and investigating the Marcinkiewicz--Zygmund inequalities on it. We note that for this specific quadrature rule on triangles, Marcinkiewicz--Zygmund inequalities may not be achievable due to the sharpness of the Markov inequality on single triangles. We then extend the 3-point quadrature rule from triangles to the triangulation of $\Omega$. For the proposed quadrature rule on polygons, we derive Marcinkiewicz--Zygmund inequalities for $1\leq p\leq\infty$. The reason why we can establish Marcinkiewicz--Zygmund inequalities on polygons, unlike the quadrature in domain points of a single triangle, is that we can use the coarsest triangulation with only the vertices of the polygon itself. This allows us to derive a Markov inequality on polygons with constants that depend only on the boundary of $\Omega$, rather than on individual triangles. As a corollary, given a set of scattered points (rather than domain points) on a single triangle with its vertices included, we can also establish Marcinkiewicz--Zygmund inequalities on it.

In Section \ref{sec:preliminaires} we collect some concepts in triangulation and Bernstein--B\'{e}zier polynomials, which are essential for our quadrature construction. Section \ref{sec:triangle} focuses on triangles, and we extend these results to polygons in Section \ref{sec:polygons}. Numerical results in Section \ref{sec:numerical} confirm the accuracy of the proposed quadrature rules for both triangles and polygons.

\section{Preliminaries}\label{sec:preliminaires}

Given any domain $\Omega$ in $\mathbb{R}^2$, we define the usual $L^{\infty}$ norm of a function $f$ by
$$
\|f\|_{\infty,\Omega}:=\textrm{ess} \sup _{x \in \Omega}|f(x)|.
$$
If $f$ is continuous on $\Omega$, we can replace the essential supremum by the maximum. For $1 \leq p<\infty$, we define the usual $L^p$ norm by
$$
\|f\|_{p, \Omega}:=\left[\int_{\Omega}|f(x)|^p d x\right]^{1 / p}.
$$
For $1\leq p\leq \infty$ and $d\geq 0$, we also define the Sobolev space 
$W^{d+1,p}(\Omega):=\left\{f\in L^p(\Omega):\|f\|_{d+1, p, \Omega}<\infty\right\}$, where the Sobolev norm is defined as 
\begin{equation*}
\|f\|_{d+1, p, \Omega}:=
\begin{cases}
\left[\sum_{k=0}^{d+1}|f|_{k, p, \Omega}^p\right]^{1 / p},\quad &1 \leq p<\infty, \\
\sum_{k=0}^{d+1}|f|_{k, \infty, \Omega}, &p=\infty
\end{cases}
\end{equation*}
with the Sobolev semi-norm defined as
\begin{equation*}
|f|_{d+1, p, \Omega}:=
\begin{cases}
\left[\sum\limits_{\nu+\mu=k}\left\|D_x^\nu D_y^\mu f\right\|_{p, \Omega}^p\right]^{1 / p}, & 1 \leq p<\infty, \\
\max\limits_{\nu+\mu=k}\left\|D_x^\nu D_y^\mu f\right\|_{\infty, \Omega}, & p=\infty.
\end{cases}
\end{equation*}
We denote by $\|\cdot\|_p$ $(1\leq p\leq \infty)$ the usual $\ell_p$ norm of vectors and matrices.

In the following, we gather key results in Bernstein--B\'{e}zier polynomials. All definitions, properties, and results provided here can be found in \cite{laischumakr}.

\textbf{Triangulation.} Given a triangle $T$, we let $|T|$ denote the length of the longest edge of $T$ and $\rho_T$ the radius of the largest disk that can be inscribed circle in $T$. The shape parameter for this triangle, denoted by $\kappa_T$, is defined as
\begin{equation*}
\kappa_T:=\frac{|T|}{\rho_T}.
\end{equation*}
We also denote by $A_T$ the area of the triangle $T$.

For any polygonal domain $\Omega\in \mathbb{R}^2$, let the collection $\triangle=\{T_1,\ldots,T_n\}$ of triangles be a \emph{triangulation} of $\Omega$. This means $\Omega = \bigcup_{i=1}^nT_n$ and any pair of triangles in $\triangle$ that intersect do so at either a common vertex or a common edge. The longest edge length in the triangulation $\triangle$, denoted by $|\triangle|$, is known as the size of the 
triangulation $\triangle$.

For functions on triangles $T$ and polygons $\Omega = \bigcup_{T\in\triangle}T$, we have the mesh-dependent relation:
\begin{equation}\label{equ:meshdependent}
\left\|D_x^\alpha D_y^\beta f\right\|_{p, \Omega}^p=\sum_{T \in \triangle}\left\|D_x^\alpha D_y^\beta f\right\|_{p, T}^p.
\end{equation}

\textbf{Bernstein--B\'{e}zier polynomials.} For a triangle $T\in \triangle$, denote by $\{(x_1,y_1),(x_2,y_2),(x_3,y_3)\}$ its vertices. For any point $(x,y)\in\Omega$, let $b_1,b_2,b_3$ be the solution of the system of linear equations
\begin{equation}\label{equ:bary}
\begin{split}
&x=b_1 x_1+b_2 x_2+b_3 x_3, \\ 
&y=b_1 y_1+b_2 y_2+b_3 y_3, \\ 
&1=b_1+b_2+b_3.
\end{split}
\end{equation}
Then $\left(b_1, b_2, b_3\right)$ is said to be the barycentric coordinates of $(x, y)$ with respect to $T$. For each $i=1,2,3$, $b_i(x, y) \geq 0$ for all $(x, y) \in T$, $b_i(x_i,y_i) = 1$, and $b_i(x,y)$ vanishes at all points on the edge of $T$ opposite to $(x_i,y_i)$. At times we denote by $v=(x,y)\in\Omega$ or $\xi\in\Omega$ a point in $\Omega$.

We use the barycentric coordinates to define the Bernstein--B\'{e}zier polynomials of degree $d$ in the form of 
\begin{equation}\label{equ:BB}
B_{i j k}^T(x, y)=\frac{d!}{i!j!k!} b_1^i b_2^j b_3^k\quad \text{for}\quad i+j+k=d.
\end{equation}
Note that the superscript $T$ denotes the dependence on the triangle $T$ rather than ``transpose.'' These polynomials satisfy 
\begin{equation}\label{equ:Brange}
0 \leq B_{i j k}^T(x,y) \leq 1\quad\text{for all}\quad(x,y)\in T
\end{equation}
and
\begin{equation}\label{equ:Bsum}
\sum_{i+j+k=d} B_{i j k}^T(x,y) \equiv 1 \quad\text{for all}\quad(x,y)\in \mathbb{R}^2.
\end{equation}
Moreover, the integral of a Bernstein--B\'{e}zier polynomial over the triangle $T$ can be evaluated in a closed form of
\begin{equation}\label{equ:Bint}
\int_T B_{i j k}^T(x, y) d x d y=\dfrac{A_T}{\binom{d+2}{2}},\quad i+j+k = d.
\end{equation}

\textbf{Polynomials and interpolation.} For Bernstein--B\'{e}zier polynomials of degree $d$, they form a basis for the space $\mathbb{P}_d$ of polynomials of (total) degree $d$. That is, for any polynomial $P \in \mathbb{P}_d$, we can write
\begin{equation}\label{equ:Bform}
P(x, y)=\sum_{i+j+k=d} c_{i j k} B_{i j k}^T(x, y).
\end{equation}
Note that the cardinality of the index set $\{(i, j, k), i+j+k=d\}$ is $\dim\mathbb{P}_d:=(d+1)(d+2) / 2$, which is the dimension of $\mathbb{P}_d$. Let $\bm{c}\in\mathbb{R}^{\dim\mathbb{P}_d}$ be the vector of coefficients $\{c_{ijk}\}_{i+j+k=d}$ in lexicographical order. This B-form representation \eqref{equ:Bform} of polynomials is stable in the following sense:
\begin{lemma}[Norm equivalence {\cite[Theorems 2.6 and 2.7]{laischumakr}}]
\label{lem:normequivalence}
Let $P$ be a polynomial in the B-form \eqref{equ:Bform} with coefficient vector $\bm{c}$. Then
$$
\frac{\|\bm{c}\|_{\infty}}{K_0} \leq\|P\|_{\infty,T} \leq\|\bm{c}\|_{\infty}
$$
and
$$
\frac{A_T^{1 / p}}{K_0}\|\bm{c}\|_p \leq\|P\|_{p, T} \leq A_T^{1 / p}\|\bm{c}\|_p
$$
for $1\leq p<\infty$, where $K_0>0$ is a constant depending only on $d$, and $K_0 = 1$ when $p=\infty$ and $d=1$.
\end{lemma}

For any $d \geq 1$, given values $\{f\left(\xi_{i j k}\right)\}_{i+j+k=d}$, where 
\begin{equation}\label{equ:domainpts}
\xi_{i j k}=\frac{i (x_1,y_1)+j (x_2,y_2)+k (x_3,y_3)}{d},\quad i+j+k=d,
\end{equation}
are domain points of degree $d$ over $T$, there exists a unique polynomial $P_f$ of degree $d$ satisfying
\begin{equation}\label{equ:interp}
P_f\left(\xi_{i j k}\right)=\sum_{l+m+n=d} c_{l m n} B_{l m n}\left(\xi_{i j k}\right)=f\left(\xi_{i j k}\right),\quad i+j+k=d.
\end{equation}
The theory of interpolation on $T$ has been well-established. We have the following special result that is sufficient for our analysis.
\begin{lemma}[Interpolation error \cite{zbMATH01234311}]\label{lem:main}
Given $f\in W^{d+1,p}(T)$ for $d\geq 1$ and $1\leq p\leq \infty$, there exists an interpolatory polynomial $P_f$ such that
\begin{equation*}
\left\|f-P_f\right\|_{p, T}\leq K_1|T|^{d+1}|f|_{d+1, p, T},
\end{equation*}
where $K_1>0$ depends only on $d$ if $p =\infty$, and $K_1>0$ depends on $d$ and $\kappa_T$ if $1\leq p<\infty$. 
\end{lemma}

We also need a Markov-type inequality for polynomials on triangles.

\begin{lemma}[Markov {\cite[Theorems 2.32]{laischumakr}}]\label{lem:markov}
There exists a constant $K_2$ depending only on $d$ such that for all polynomials $P \in \mathbb{P}_d$,
\begin{equation}\label{equ:Markov}
\left\|D_x^\alpha D_y^\beta P\right\|_{p, T} \leq \frac{K_2}{\rho_T^{\alpha+\beta}}\|P\|_{p, T}, \quad 0 \leq \alpha+\beta \leq d,
\end{equation}
for all $1 \leq p \leq \infty$, where $K_2 = K_0(2d)^{\alpha+\beta}$.
\end{lemma}

\section{Main results}

We first investigate Marcinkiewicz--Zygmund inequalities on triangles and then extend them to polygons.

\subsection{On triangles}\label{sec:triangle}

Let $\bm{B}\in\mathbb{R}^{\dim\mathbb{P}_d\times \dim{\mathbb{P}_d}}$ be the matrix of entries $B^T_{lmn}(\xi_{ijk})$, with $ijk$ being row index and $lmn$ column index, where $\{B^T_{lmn}\}$ are Bernstein--B\'{e}zier polynomials \eqref{equ:BB}, $\{\xi_{ijk}\}_{i+j+k=d}$ are domain points \eqref{equ:domainpts}, and the rows and columns of $\bm{B}$ are in lexicographical order, respectively. We are interested in quadrature rules in these domain points, that is, 
\begin{equation}\label{equ:quadT}
\sum_{i+j+k =d}w_{ijk}f(\xi_{ijk})\approx \int_Tf(x,y)dxdy
\end{equation}
with some quadrature weights $\{w_{ijk}\}_{i+j+k=d}$.

\begin{theorem}[Quadrature on triangles: exactness]\label{thm:quad}
Let $T$ be a triangle in $\mathbb{R}^2$ and $d\geq 1$ an integer. Let
\begin{equation}\label{equ:weightsT}
w_{ijk} := \frac{A_T}{\binom{d+2}{2}}\sigma_{ijk}(\bm{B}^{-1}),\quad i+j+k=d,
\end{equation}
where $\sigma_{ijk}(\bm{B}^{-1})$ is the $ijk$-th (lexicographical order) column sum of $\bm{B}^{-1}$. Then the quadrature rule \eqref{equ:quadT} in quadrature points \eqref{equ:domainpts} with quadrature weights \eqref{equ:weightsT} is exact for all polynomials of degree at most $d$, that is,
$$\sum_{i+j+k=d}w_{ijk}\chi(\xi_{ijk}) = \int_T\chi(x,y)dxdy\quad \forall \chi\in\mathbb{P}_d.$$
\end{theorem}

\begin{proof}
For any $\chi\in\mathbb{P}_d$, let $P$ be an interpolatory polynomial in $\mathbb{P}_d$ such that
\begin{equation}\label{equ:exactinterp}
\chi(\xi_{ijk}) = P(\xi_{ijk}) = \sum_{l+m+n=d} c_{lmn}B^T_{lmn}(\xi_{ijk}),\quad i+j+k=d.
\end{equation}
In fact we have $\chi(x,y) = P(x, y)$. Then we have
$$\int_T\chi(x,y)dxdy = \int_{T}P(x,y)dxdy = \sum_{l+m+n=d}c_{lmn}\int_TB^T_{lmn}(x,y)dxdy.$$
Let $\bm{x}\in\mathbb{R}^{\dim\mathbb{P}_d}$ be with entries $[\chi(\xi_{ijk})]_{i+j+k=d}$ in lexicographical order. The interpolation conditions \eqref{equ:exactinterp} are equivalent to a system of linear equations $\bm{B}\bm{c} = \bm{x}$. Thus we have 
$$\bm{c} = \bm{B}^{-1}\bm{x},$$
where the invertibility of $B$ stems from the the existence and uniqueness result of polynomial interpolation on $T$. Together with the closed-form integral \eqref{equ:Bint} of $B^T_{lmn}$, we obtain the weights \eqref{equ:weightsT} such that quadrature rule \eqref{equ:quadT} is exact for $\chi$.
\end{proof}

\begin{theorem}[Quadrature on triangles: error analysis]\label{thm:quaderrorT}
Let $T$ be a triangle in $\mathbb{R}^2$ and $d\geq 1$ an integer. For any $f\in W^{d+1,p}(T)$ with $1\leq p\leq \infty$, the quadrature rule \eqref{equ:quadT} in domain points \eqref{equ:domainpts} with quadrature weights \eqref{equ:weightsT} satisfies
\begin{equation*}
\left|\int_{T}f(x,y)dxdy - \sum_{i+j+k=d}w_{ijk}f(\xi_{ijk})\right|\leq K_1|A_T|^{1/q}|T|^{d+1}|f|_{d+1, p, T},
\end{equation*}
where $1/p+1/q =1$ and $K_1>0$ is the constant given in Lemma \ref{lem:main}.
\end{theorem}

\begin{proof}
Let $P\in\mathbb{P}_d$ be the polynomial interpolant of $f$. Then the quadrature rule
$$\int_Tf(x,y)dxdy\approx \int_{T}P(x,y)dxdy = \sum_{i+j+k=d}w_{ijk}f(\xi_{ijk})$$
satisfies
\begin{equation*}\begin{split}
\left|\int_{T}f(x,y)dxdy - \sum_{i+j+k=d}w_{ijk}f(\xi_{ijk})\right|
&\leq \int_T\left|f(x,y)-P(x,y)\right|dxdy\\
&\leq |A_T|^{1/q}\|f-P\|_{p,T}.
\end{split}\end{equation*}
Together with Lemma \ref{lem:main}, the desired error estimate is obtained.
\end{proof}

Figure \ref{fig:location} illustrates the locations of domain points \eqref{equ:domainpts} for equilateral triangles with $d=1,2,\ldots,6$, along with their corresponding weights \eqref{equ:weightsT}. Note that the proposed set of weights \eqref{equ:weightsT} may include negative values. In our analysis of Marcinkiewicz--Zygmund inequalities, we are particularly interested in quadrature rules with positive weights. As shown in Figure \ref{fig:location}, the weights \eqref{equ:weightsT} are positive for $d=1,3,5$.

\begin{figure}[htbp]
\centering
\includegraphics[width = 0.45\textwidth]{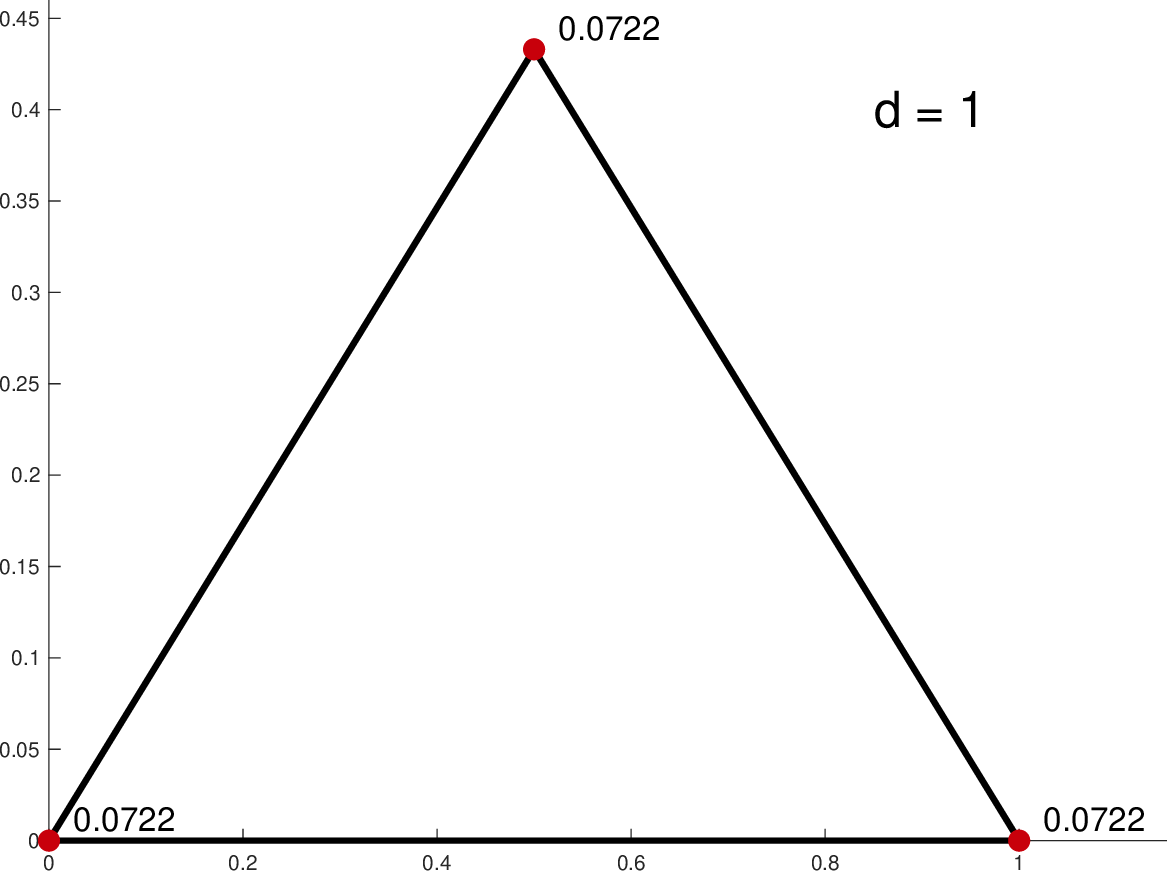}\quad\quad\includegraphics[width = 0.45\textwidth]{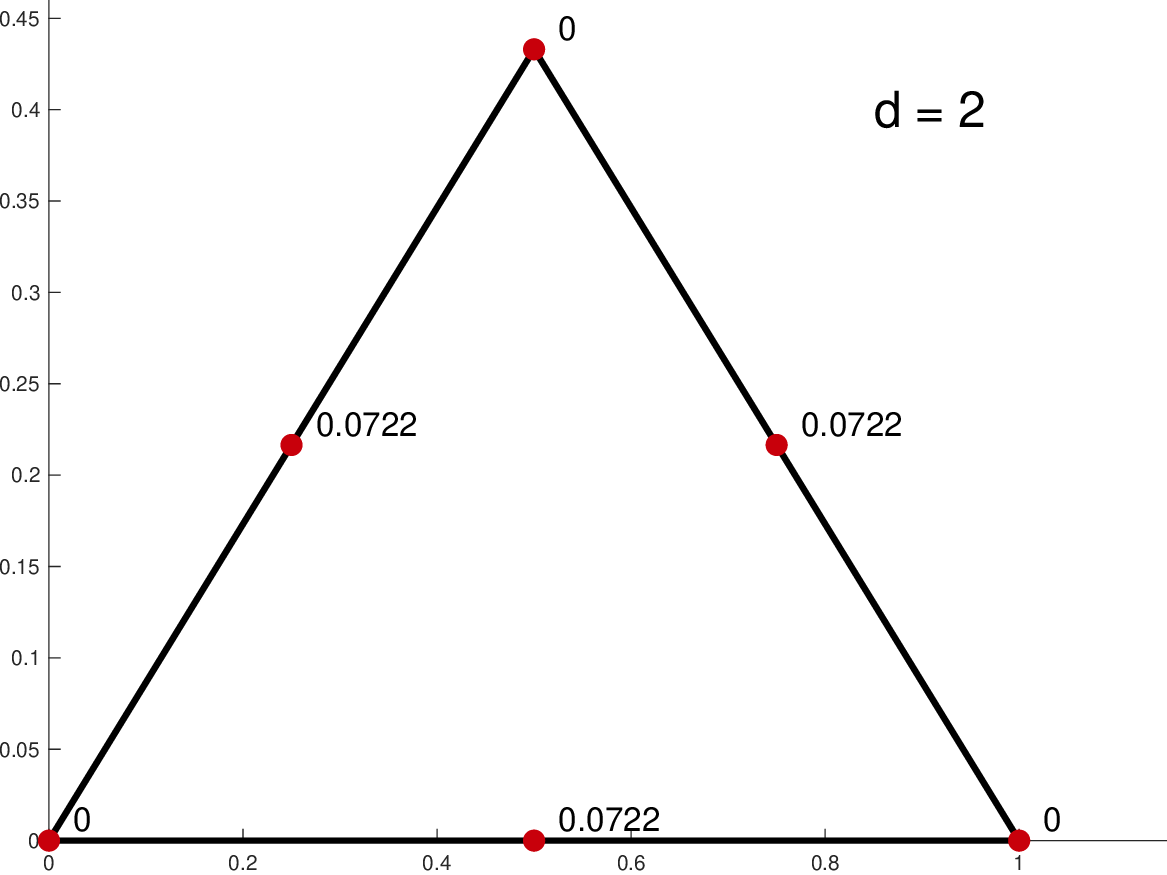}\\~\\~\\
\includegraphics[width = 0.45\textwidth]{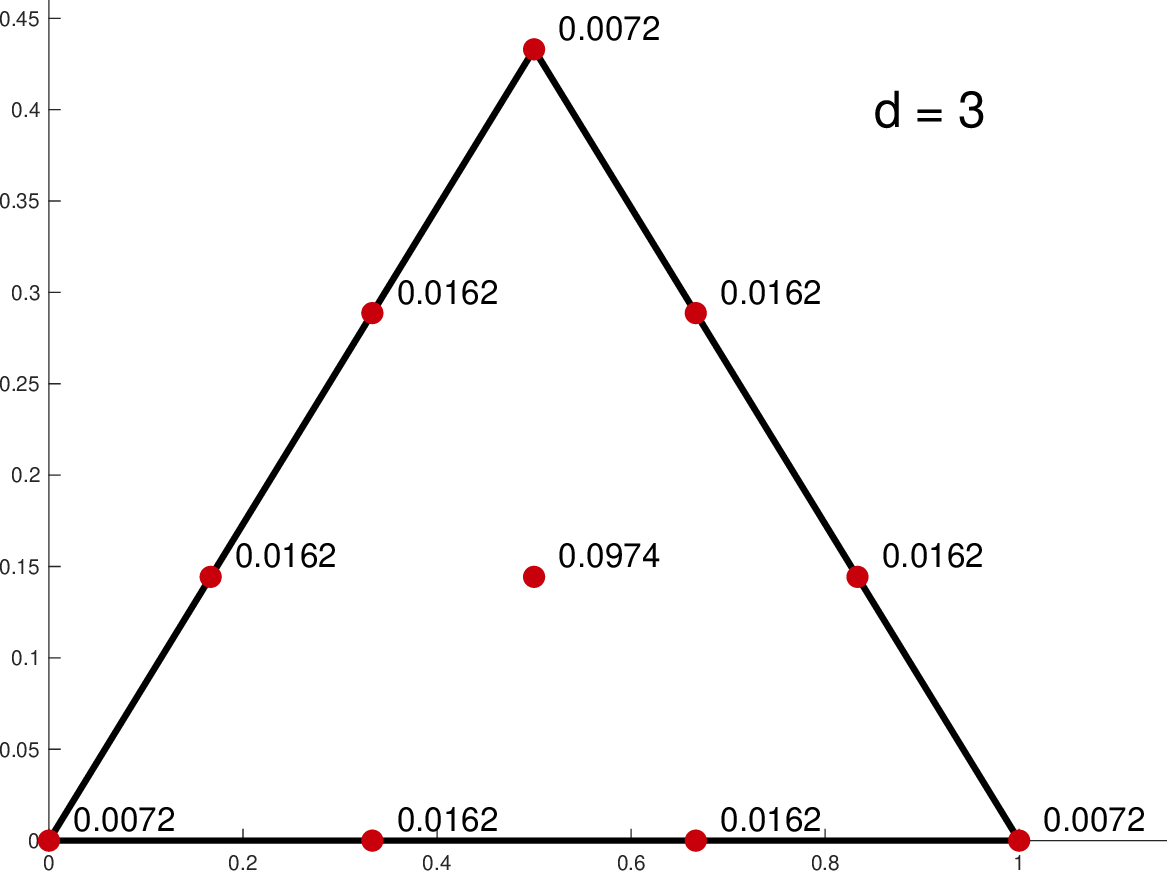}\quad\quad\includegraphics[width = 0.45\textwidth]{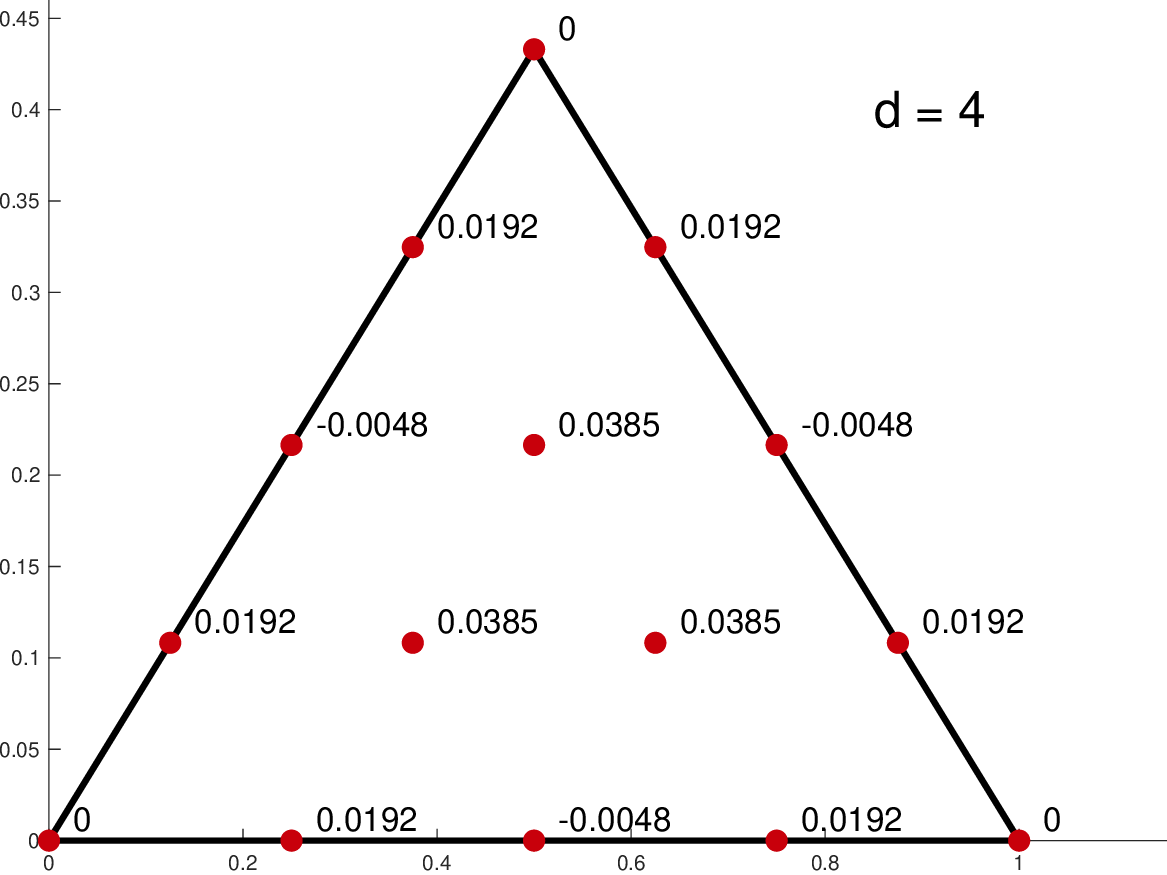}\\~\\~\\
\includegraphics[width = 0.45\textwidth]{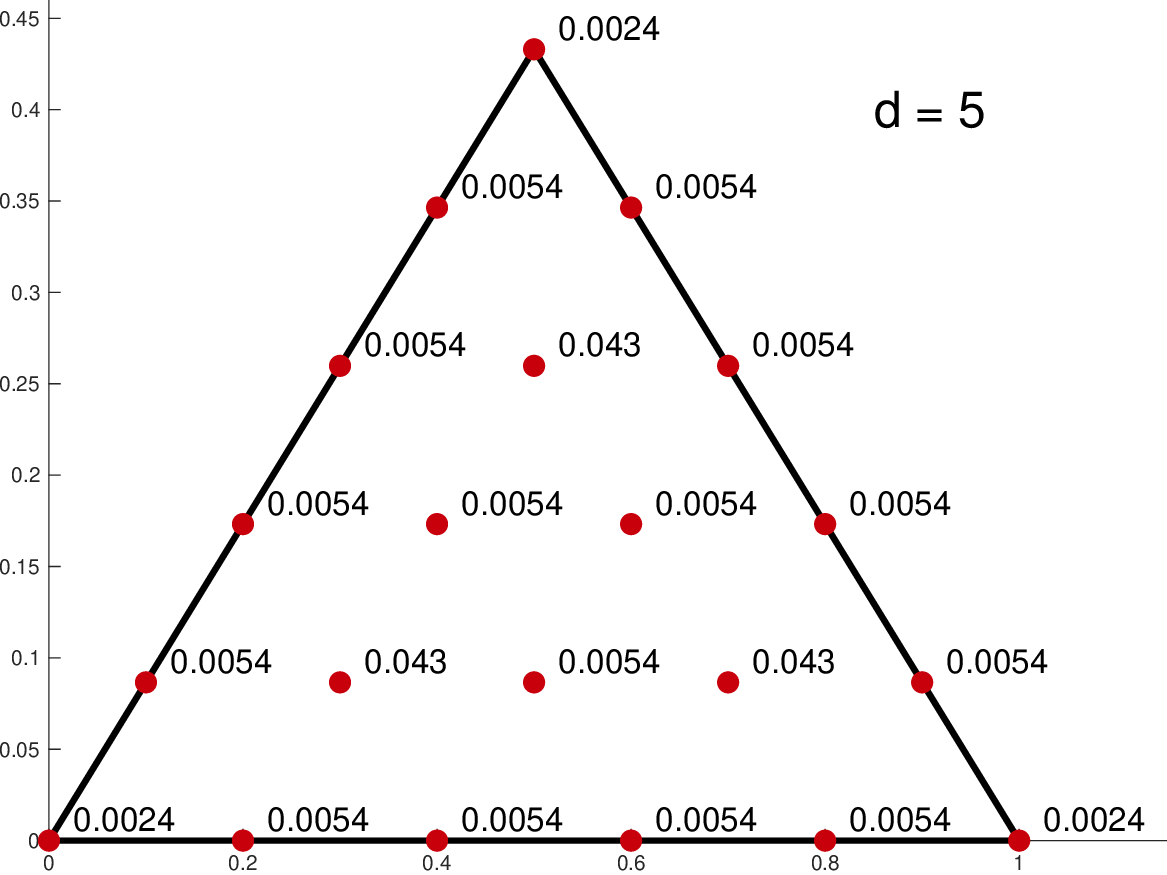}\quad\quad\includegraphics[width = 0.45\textwidth]{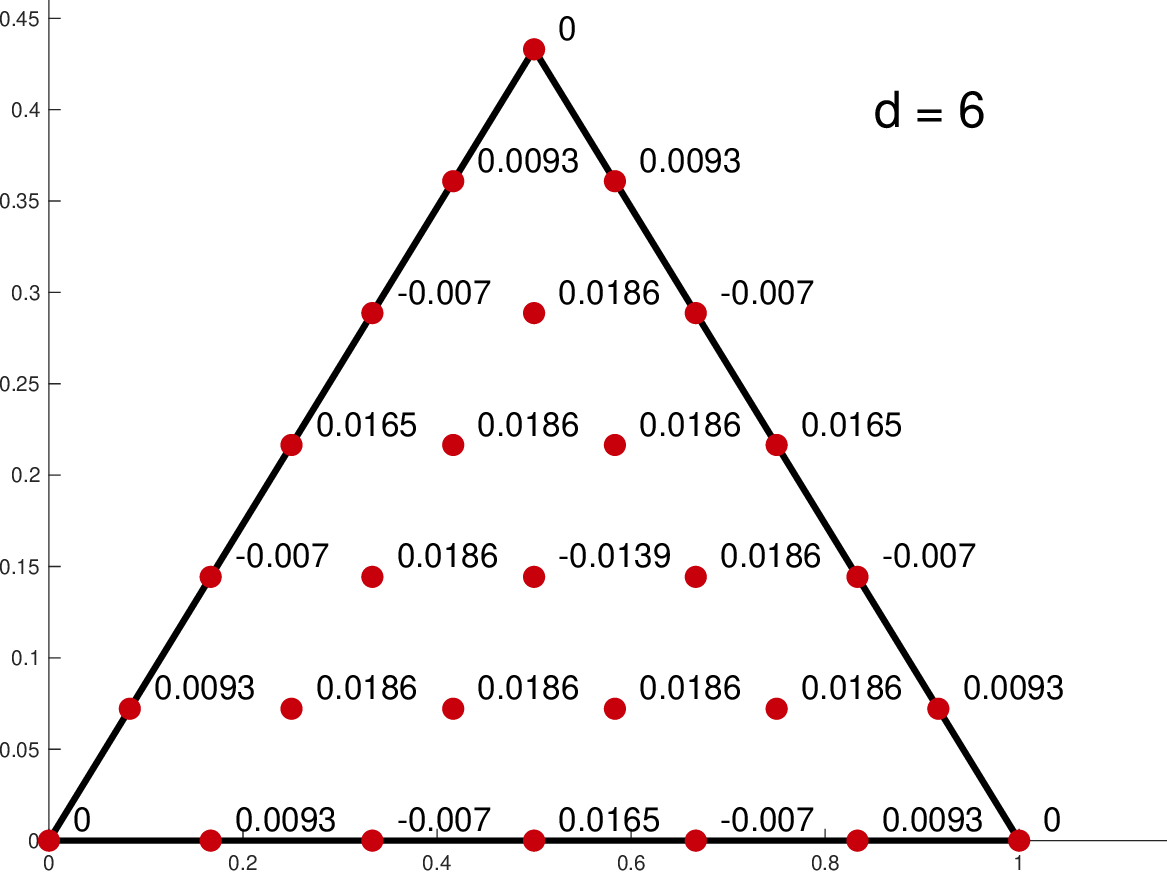}\\
\caption{Locations of domain points and corresponding weights for various $d$.}\label{fig:location}
\end{figure}

For triangles other than equilateral triangles, note that barycentric coordinates are solutions to the linear system \eqref{equ:bary}. Any triangle can be considered as a linear transformation of an equilateral triangle. Applying such a linear, full-rank operator to both sides of the linear system \eqref{equ:bary} does not alter the barycentric coordinates. Therefore, we can claim that for any triangle, the weights \eqref{equ:weightsT} are positive for $d=1,3,5$. Specifically, for $d=1$, we have
\begin{equation}\label{equ:Bidentity}
\bm{B}=
\begin{bmatrix} B_{001}^T(\xi_{001}) & B_{010}^T(\xi_{001}) & B_{100}^T(\xi_{001})
\\ B_{001}^T(\xi_{010}) & B_{010}^T(\xi_{010})) & B_{100}^T(\xi_{010})\\ 
B_{001}^T(\xi_{100}) & B_{010}^T(\xi_{100}) & B_{100}^T(\xi_{100}) \end{bmatrix}
= \begin{bmatrix}
1 & 0 & 0\\
0 & 1 & 0\\
0 & 0 & 1
\end{bmatrix}\in\mathbb{R}^3.
\end{equation}
is an identity matrix. This observation follows immediately from the fact that $b_i(x_i,y_i)$ is 1 and vanishes on the edge of $T$ opposite to $(x_i,y_i)$ for each $i = 1,2,3$.

Thus we focus on the case of $d=1,3,5$ in the following discussion.

\begin{theorem}[MZ on triangles, I]\label{thm:MZ1}
Let $T$ be a triangle in $\mathbb{R}^2$ and $d = 1,3,5$. For any polynomial $\chi\in\mathbb{P}_N$ with $N\leq d$, the quadrature rule \eqref{equ:quadT} in domain points \eqref{equ:domainpts} with weights \eqref{equ:weightsT} satisfies the following Marcinkiewicz--Zygmund inequalities:
\begin{equation}\label{equ:MZIp}\begin{split}
C_1(p,d)\int_T|\chi(x,y)|^pdxdy&\leq \sum_{i+j+k=d}w_{ijk}|\chi(\xi_{ijk})|^p\\
&\leq C_2(p,d)K_0^p\int_T|\chi(x,y)|^pdxdy
\end{split}\end{equation}
for $1\leq p < \infty$, and
\begin{equation}\label{equ:MZIinf}
C_3(d)\max_{i+j+k=d}|\chi(\xi_{ijk})|\leq \|\chi\|_{\infty,T}\leq K_0\max_{i+j+k=d}|\chi(\xi_{ijk})|
\end{equation}
for $p=\infty$, where $C_1,C_2>0$ are constants depending only on $d$ and $p$, $C_3>0$ are constants depending only on $d$, and $K_0>0$ is the constant given in Lemma \ref{lem:normequivalence}. All these constants are independent of $N$.
\end{theorem}

\begin{proof}
For any $\chi\in\mathbb{P}_N$ with $N\leq d$, let $P$ be an interpolatory polynomial in $\mathbb{P}_d$ interpolating $\chi$ at domain points \eqref{equ:domainpts}. From the proof of Theorem \ref{thm:quad}, we have
$$\chi(x,y) = P(x,y) = \sum_{l+m+n=d}c_{lmn}B^T_{lmn}(x,y)$$
with $\bm{Bc} = \bm{x}$. For $d=1,3,5$, note that $w_{ijk}>0$ and hence $\sigma_{ijk}(\bm{B}^{-1})>0$ for $i+j+k=d$.

We first prove the case of $1\leq p<\infty$. Applying the norm equivalence in Lemma \ref{lem:normequivalence} to $\chi$, we obtain
$$\frac{A_T^{1 / p}}{K_0}\|\bm{c}\|_p \leq\|\chi\|_{p, T} = \|P\|_{p, T} \leq A_T^{1 / p}\|\bm{c}\|_p.$$
From $\|\chi\|_{p,T}^p \leq A_T\|\bm{c}\|_{p}^p$ we have
\begin{equation*}
\begin{split}
\|\chi\|_{p,T}^p 
&\leq A_T\sum_{i+j+k=d}|\sigma_{ijk}(\bm{B}^{-1})|^p|\chi(\xi_{ijk})|^p \\
& = A_T\sum_{i+j+k=d} \left[\frac{\binom{d+2}{2}w_{ijk}}{A_T}\right]^p|\chi(\xi_{ijk})|^p\\
& \leq \frac{\binom{d+2}{2}^p}{|A_T|^{p-1}}\left[\max_{i+j+k=d}w_{ijk}\right]^{p-1}\sum_{i+j+k=d}w_{ijk}|\chi(\xi_{ijk})|^p\\
& = \binom{d+2}{2}\left[\max_{i+j+k=d}\sigma_{ijk}(\bm{B}^{-1})\right]^{p-1}\sum_{i+j+k=d}w_{ijk}|\chi(\xi_{ijk})|^p,
\end{split}\end{equation*}
implying 
$$\sum_{i+j+k=d}w_{ijk}|\chi(\xi_{ijk})|^p\geq \frac{1}{\binom{d+2}{2}\left(\max_{i+j+k=d}\sigma_{ijk}(\bm{B}^{-1})\right)^{p-1}}\|\chi\|_{p,T}^p.$$
Conversely, from $\|\chi\|_{p,T}^p\geq \frac{A_T}{K_0^p}\|\bm{c}\|_{p}^p$ we have
\begin{equation*}
\begin{split}
\|\chi\|_{p,T}^p 
& \geq \frac{\binom{d+2}{2}}{K_0^p}\left[\min_{i+j+k=d}\sigma_{ijk}(\bm{B}^{-1})\right]^{p-1}\sum_{i+j+k=d}w_{ijk}|\chi(\xi_{ijk})|^p,
\end{split}\end{equation*}
implying
$$\sum_{i+j+k=d}w_{ijk}|\chi(\xi_{ijk})|^p\leq \frac{K_0^p}{\binom{d+2}{2}\left(\min_{i+j+k=d}\sigma_{ijk}(\bm{B}^{-1})\right)^{p-1}}\|\chi\|_{p,T}^p.$$
Combing both inequalities, we have the Marcinkiewicz--Zygmund inequality \eqref{equ:MZIp} for $1\leq p<\infty$.

For $p=\infty$, the norm equivalence in Lemma \ref{lem:normequivalence} indicates
$$\frac{\|\bm{c}\|_{\infty}}{K_0} \leq\|\chi\|_{\infty,T} \leq\|\bm{c}\|_{\infty}.$$
Thus from $\|\chi\|_{\infty,T}\leq \|\bm{c}\|_{\infty}$ we have

$$\|\chi\|_{\infty,T}\leq\max_{i+j+k=d}|\sigma_{ijk}(\bm{B}^{-1})||\chi(\xi_{ijk})|\leq \left[\max_{i+j+k=d}\sigma_{ijk}(\bm{B}^{-1})\right]\max_{i+j+k=d}|\chi(\xi_{ijk})|$$
and hence
$$\max_{i+j+k=d}|\chi(\xi_{ijk})|\geq \frac{1}{\max_{i+j+k=d}\sigma_{ijk}(\bm{B}^{-1})}\|\chi\|_{\infty,T}.$$
From $\|\chi\|_{\infty,T} \geq \frac{\|\bm{c}\|_{\infty}}{K_0}$, recalling from \eqref{equ:Bsum} that $\|\bm{B}\|_{\infty} = 1$ as $B^T_{lmn}(\xi_{ijk})\geq0$, we have
$$|\chi\|_{\infty,T}\geq\frac{1}{K_0}\frac{\|\bm{x}\|_{\infty}}{\|\bm{B}\|_{\infty}} = \frac{1}{K_0}\max_{i+j+k=d}|\chi(\xi_{ijk})|$$
and
$$\max_{i+j+k=d}|\chi(\xi_{ijk})|\leq K_0\|\chi\|_{\infty,T}.$$ 
Combing both inequalities, we have the Marcinkiewicz--Zygmund inequality \eqref{equ:MZIinf} for $p=\infty$.
\end{proof}

\begin{theorem}[MZ-like on triangles, II]\label{thm:MZ2}
Let $T$ be a triangle in $\mathbb{R}^2$ and let $d\geq 1$ be an integer. For any polynomial $\chi\in\mathbb{P}_N$ with $N> d$, the quadrature rule \eqref{equ:quadT} with weights \eqref{equ:weightsT} satisfies the following estimates
\begin{equation}\label{equ:MZIp2}
\begin{split}
C_1(p,d)(\|\chi\|_{p,T}-K_1|T|^{d+1}|\chi|_{p,d+1,T})^p&\leq \sum_{i+j+k=d}w_{ijk}|\chi(\xi_{ijk})|^p\\
&\leq C_2(p,d)K_0^p(\|\chi\|_{p,T}+K_1|T|^{d+1}|\chi|_{p,d+1,T})^p
\end{split}\end{equation}
for $1\leq p < \infty$, and
\begin{equation}\label{equ:MZIinf2}
\begin{split}
C_3(d)&\left(\|\chi\|_{\infty,T}-K_1|T|^{d+1}|\chi|_{\infty,d+1,T}\right)\leq\max_{i+j+k=d}|\chi(\xi_{ijk})| \\
&\leq K_0\left(\|\chi\|_{\infty,T}+K_1|T|^{d+1}|\chi|_{\infty,d+1,T}\right)\end{split}\end{equation}
for $p=\infty$, where $C_1,C_2,C_3>0$ are constants given in Theorem \ref{thm:MZ1}, $K_0,K_1>0$ are constants given in Lemmata \ref{lem:normequivalence} and \ref{lem:main}, respectively. All these constants are independent of $N$.
\end{theorem}

\begin{proof}
The proof immediately follows from that of Theorem \ref{thm:MZ1}, with the only difference that $P(x,y)=\chi(x,y)$ no longer holds. Instead, by Lemma \ref{lem:main} we have
\begin{equation}\label{equ:stepwithmarkov}
\|\chi-P\|_{p,T}\leq K_1|T|^{d+1}|\chi|_{d+1,p,T}.
\end{equation}

For $1\leq p<\infty$, both $\|\chi\|_{p,T}-\|P\|_{p,T}$ and $\|P\|_{p,T}-\|\chi\|_{p,T}$ are bounded by $K_1|T|^{d+1}|\chi|_{d+1,p,T}$. Thus, by the norm equivalence in Lemma \ref{lem:normequivalence}, 
$$\|\chi\|_{p,T}-K_1|T|^{d+1}|\chi|_{p,d+1,T}\leq \|P\|_{p,T}\leq A_T^{1/p}\|\bm{c}\|_{p}$$
and 
$$\|\chi\|_{p,T}+K_1|T|^{d+1}|\chi|_{p,d+1,T}\geq \|P\|_{p,T}\geq \frac{A_T^{1/p}}{K_0}\|\bm{c}\|_{p},$$
then the estimate \eqref{equ:MZIp2} for $1\leq p<\infty$ follows from applying the argument for $\|\bm{c}\|_p$ for $1\leq p<\infty$ as in the proof of Theorem \ref{thm:MZ1}.

For $p=\infty$, similarly we have 
$$\|\chi\|_{\infty,T}-K_1|T|^{d+1}|\chi|_{\infty,d+1,T}\leq \|P\|_{\infty,T}\leq \|\bm{c}\|_{\infty}$$
and 
$$\|\chi\|_{p,\infty}+K_1|T|^{d+1}|\chi|_{\infty,d+1,T}\geq \|P\|_{\infty,T}\geq \frac{1}{K_0}\|\bm{c}\|_{\infty},$$
thus we have the estimate \eqref{equ:MZIinf2} following the argument for $p=\infty$ in the proof of Theorem \ref{thm:MZ1}.
\end{proof}

\begin{remark}\label{rem:limitation}
Estimates \eqref{equ:MZIp2} and \eqref{equ:MZIinf2} are not Marcinkiewicz--Zygmund inequalities. Of course we could incorporate \eqref{equ:stepwithmarkov} with the Markov inequality in Lemma \ref{lem:markov} to further obtain
\begin{equation*}
\|\chi-P\|_{p,T}\leq K_1K_2\frac{|T|^{d+1}}{\rho_{T}^{d+1}}\|\chi\|_{p,T}  = K_1K_2\kappa_T^{d+1}\|\chi\|_{p,T},
\end{equation*}
and then show 
$\sum_{i+j+k=d}w_{ijk}|\chi(\xi_{ijk})|^p$ and $\max_{i+j+k=d}|\chi(\xi_{ijk})|$ can be bounded in terms of $\|\chi\|_{p,\Omega}$ in the Marcinkiewicz--Zygmund sense. However, in that case, the constants $1+ K_1K_2\kappa_T^{d+1}$ and $1- K_1K_2\kappa_T^{d+1}$ in both bounds can not be shrunk toward $1$ by shrinking the triangle (or refining the triangulation on polygons; see later analysis on polygons), limiting the resulted estimates from practical applications. Note that the shape parameter $\kappa_T>1$ does not reduce as triangles shrinking: an equilateral triangle has $\kappa_T=2 \sqrt{3}$, and any other triangle has a larger shape parameter. This situation was caused by the $1/\rho_T$ factor in the Markov inequality in Lemma \ref{lem:markov}, and $|T|/\rho_T$ gives $\kappa_T$. In other words, if this Markov inequality is sharp, then we may not have Marcinkiewicz--Zygmund inequalities for the quadrature \eqref{thm:quad} with controllable 
Marcinkiewicz--Zygmund constants.
\end{remark}

The case of $d=1$ is of significant interest in our subsequent investigation of polygons. In this scenario, $\bm{B}$ and hence $\bm{B}^{-1}$ are identity matrices, see \eqref{equ:Bidentity}, and consequently, 
$$\max_{i+j+k=d}\sigma_{ijk}(\bm{B}^{-1}) = \min_{i+j+k=d}\sigma_{ijk}(\bm{B}^{-1})=1.$$
Moreover in this case, for constants in above theorems, we have
$C_1=1/3$, $C_2=1/3$, and $C_3=1$.
Referring to Lemma \ref{lem:normequivalence} we have $K_0=1$ if $p=\infty$.

\subsection{On polygons}\label{sec:polygons}

Let $\Omega\subset\mathbb{R}^2$ be a regular or irregular polygon. Given a set of scattered points $\{(x_j,y_j)\}_{j=1}^m\subset\Omega$ with vertices of $\Omega$ included, we generate a triangulation $\triangle$ of $\Omega$ in the sense that $\Omega =  \bigcup_{T\in\triangle}T$ and these points serve as vertices of these triangles $T$. We propose the following quadrature rule for numerical integration on $\Omega$: 
\begin{equation}\label{equ:quadP}
\sum_{j=1}^mw_jf(x_j,y_j)\approx\int_{\Omega}f(x,y)dxdy,
\end{equation}
where 
\begin{equation}\label{equ:weightsP}
w_j:=\sum_{T:(x_j,y_j)\in T}\frac{A_T}{3}.
\end{equation}
In other words, for any quadrature point $(x_j,y_j)$, the corresponding weight $w_j$ is one third of the sum of the areas of all triangles with $(x_j, y_j)$ being a vertex. The intuition behind \eqref{equ:weightsP} stems from the weights \eqref{equ:weightsT} on triangles with $d=1$. Note that in this scenario on polygons, only scattered points are available as quadrature points, and additional domain points within each $T$ are not available.

As noted in Remark \ref{rem:limitation}, a limitation on triangles is that applying the Markov inequality to the semi-norm $|\chi|_{p,d+1,\Omega}$ introduces a factor of $1/\rho_{T}^{2}$, and $|T|^2/\rho_{T}^{2}$ results in the squared shape parameter $\kappa_T^2$. However, on polygons, we have the privilege of letting the unfavored factor be dependent of the polygonal $\Omega$ rather than any single triangle in $\triangle$. 

\begin{lemma}[Markov on polygons]\label{lem:Markovpolygons}
For all $1\leq p\leq\infty$, there exists a constant $C_5$ depending on the boundary of $\Omega$ such that for all polynomials $P \in \mathbb{P}_d$,
\begin{equation}
\left\|D_x^\alpha D_y^\beta P\right\|_{p, \Omega} \leq C_5d^{\alpha+\beta}\|P\|_{p, \Omega}, \quad 0 \leq \alpha+\beta \leq d.
\end{equation}
\end{lemma}

\begin{proof}
Let $\triangle_0$ be a triangulation of $\Omega$ consisting of triangles with vertices being the vertices of $\Omega$ only. For $1\leq p<\infty$, by the Markov inequality on triangles we have 
$$\left\|D_x^\alpha D_y^\beta P\right\|_{p, \Omega}^p=\sum_{T'\in\triangle_0}\left\|D_x^\alpha D_y^\beta P\right\|_{p, T'}^p\leq \left[\frac{K_0(2d)^{\alpha+\beta}}{|\min_{T'\in\triangle_0}\rho_{T'}|^{\alpha+\beta}}\right]^p\sum_{T'\in\triangle_0}\left\|P\right\|_{p, T'}^p$$
and 
$$\sum_{T'\in\triangle_0}\left\|P\right\|_{p, T'}^p = \|\chi\|_{p,\Omega}^p,$$
where $\min_{T'\in\triangle_0}\rho_{T'}$ depends on the boundary of $\Omega$ but not on the size $|\triangle|$ of the triangulation $\triangle$.

For $p = \infty$, let $T_0\in\triangle_0$ be the triangle where $|D_x^\alpha D_y^\beta P|$ attains its maximum. Then
$$\left\|D_x^\alpha D_y^\beta P\right\|_{\infty, \Omega}
=\left\|D_x^\alpha D_y^\beta P\right\|_{\infty, T_0}\leq \frac{K_0(2d)^{\alpha+\beta}}{\rho_{T_0}^{\alpha+\beta}}\|P\|_{\infty,T_0}\leq\frac{K_0(2d)^{\alpha+\beta}}{\rho_{T_0}^{\alpha+\beta}}\|P\|_{\infty,\Omega},
$$
and note that $1/\rho_{T_0}\leq 1/\min_{T'\in\triangle_0}\rho_{T'}$, which is a constant depending only on the boundary of $\Omega$.
\end{proof}

For the geometry of the triangulation $\triangle$, we let $\gamma_{\triangle}\geq 1$ be the ratio of the longest edge length in the triangulation over $\rho_{T_{\min}}$, where $T_{\min}$ is the smallest-area triangle in $\triangle$. Then we have the following estimates:
\begin{lemma}\label{lem:number}
Let $\#\triangle$ be the number of triangles in $\triangle$. Then there exists constant $C_7>0$ depending only on boundary of $\Omega$ such that 
$$|\#\triangle|^{1-1/p}|\triangle|^2 \leq C_7|\triangle|^{2/p}.$$
\end{lemma}
\begin{proof}
It follows from $|\#\triangle|\leq {|\Omega|}/{A_{T_{\min}}}\leq {|\Omega|}/{\pi\rho_{T_{\min}}^2}\leq C_6/\rho_{T_{\min}}^2$
and
$$|\#\triangle|^{1-1/p}|\triangle|^2 
\leq C_6^{1-1/p}\frac{|\triangle|^{2-2/p}}{\rho_{T_{\min}}^{2-2/p}}|\triangle|^{2/p} = C_6^{1-1/p}\gamma_{\triangle}^{2-2/p}|\triangle|^{2/p},
$$
where $C_6$, $C_7$, and $\gamma_{\triangle}^{2-2/p}$ depend on the boundary of $\Omega$.
\end{proof}

We first investigate the exactness degree and error bound of the proposed quadrature rule \eqref{equ:quadP} on polygons.

\begin{theorem}[Quadrature on polygons]\label{thm:quadP}
Let $\Omega$ be a polygon in $\mathbb{R}^2$ and $\{(x_j,y_j)\}_{j=1}^m$ on $\Omega$ a set of scattered points in $\Omega$ with vertices of $\Omega$ included. Let $\triangle$ be a triangulation of $\Omega$ generated from $\{(x_j,y_j)\}_{j=1}^m$. Then the quadrature rule \eqref{equ:quadP} in quadrature points $\{(x_j,y_j)\}_{j=1}^m$ with quadrature weights \eqref{equ:weightsP} is exact for all constant and linear polynomials, that is,
$$\sum_{j=1}^mw_{j}\chi(x_j,y_j) = \int_{\Omega}\chi(x,y)dxdy\quad \forall \chi\in\mathbb{P}_1.$$
Besides, for any $f\in W^{2,p}(\Omega)$ with $1\leq p\leq \infty$, the quadrature rule \eqref{equ:quadP} satisfies
$$\left|\int_{\Omega}f(x,y)dxdy - \sum_{j=1}^mw_{j}f(x_j,y_j)\right|\leq K_1C_7\max_{T\in\triangle}|A_T|^{1/q}|f|_{2, p, \Omega},$$
where $1/p+1/q =1$ and $K_1,C_7>0$ are constants given in Lemmata \ref{lem:main} and \ref{lem:number}, respectively.
\end{theorem}
\begin{proof}
This is an immediate consequence of Theorems \ref{thm:quad} and \ref{thm:quaderrorT}, and the fact of $\Omega = \bigcup_{T\in\triangle}T$. Specifically, note that 
$$\sum_{j=1}^mw_{j}\chi(x_j,y_j) = \sum_{T\in\triangle} \sum_{i+j+k=1}w_{ijk}\chi(\xi_{ijk})$$
and 
$$\int_{\Omega}\chi(x,y)dxdy = \sum_{T\in\triangle} \int_{T}\chi(x,y)dxdy.$$
Thus the exactness result follows from Theorem \ref{thm:quad} for $d=1$. Similarly, note that 
$$\int_{\Omega}f(x,y)dxdy - \sum_{j=1}^mw_{j}f(x_j,y_j) = \sum_{T\in\triangle} \left[\int_Tf(x,y)dxdy - \sum_{i+j+k=1}w_{ijk}f(\xi_{ijk})\right].$$
The error estimate then follows from Theorem \ref{thm:quaderrorT} that
$$\left|\int_{\Omega}f(x,y)dxdy - \sum_{j=1}^mw_{j}f(x_j,y_j)\right|\leq |\#\triangle|K_1\max_{T\in\triangle}|A_T|^{1/q}|\triangle|^{2}|f|_{2, p, \Omega},$$
and Lemma \ref{lem:number} that $|\#\triangle||\triangle|^2 \leq C_7$.
\end{proof}

We then establish the following Marcinkiewicz--Zygmund inequalities for any polynomial $\chi\in\mathbb{P}_N$ with $N>1$.

\begin{theorem}[MZ on polygons]\label{thm:MZpolygon}
Let $\Omega$ be a polygon in $\mathbb{R}^2$ and $\{(x_j,y_j)\}_{j=1}^m$ a set of scattered points in $\Omega$ with vertices of $\Omega$ included. Let $\triangle$ be a triangulation of $\Omega$ generated from $\{(x_j,y_j)\}_{j=1}^m$. The quadrature rule \eqref{equ:quadP} in quadrature points $\{(x_j,y_j)\}_{j=1}^m$ with quadrature weights \eqref{equ:weightsP} satisfies following Marcinkiewicz--Zygmund inequalities:

(a) Let $1\leq p<\infty$. There exists a constant $c_1>0$ depending only on $p$ and the boundary of $\Omega$ such that for any $N>0$, if $N^2|\triangle|^{2/p}<1/c_1$, then
for $\chi\in\mathbb{P}_N$,
\begin{equation}\label{equ:MZLp3}
\left|\sum_{j=1}^mw_j|\chi(x_j,y_j)|^p-\int_{\Omega}|\chi(x,y)|^pdxdy\right|\leq \eta \int_{\Omega}|\chi(x,y)|^pdxdy,
\end{equation}
where 
$$\eta = \max\left\{\frac{K_0^p}{3}\left[|\#\triangle|^{1-1/p}+c_1N^2 |\triangle|^{2/p}\right]^p-1,
1-\frac{(1 - c_1N^2|\triangle|^{2/p})^p}{3|\#\triangle|^{p-1}}
\right\}.$$

(b) There exists a constant $c_2>0$ depending on the boundary of $\Omega$ such that for any $\eta>0$, if 
$$N\leq c_2\min\left\{\frac{1}{|\triangle|},\frac{1}{|\triangle|^{1/2}}\right\}\eta^{1/2},$$ 
then for $\chi\in\mathbb{P}_N$, 
\begin{equation}\label{equ:MZpolygoninf}
\left|\max_{1\leq j \leq m}|\chi(x_j,y_j)| -\|\chi\|_{\infty,\Omega}\right|\leq \eta\|\chi\|_{\infty,\Omega}.
\end{equation}
\end{theorem}

\begin{proof}
For $p=\infty$, the estimate \eqref{equ:MZIp2} in Theorem \ref{thm:MZ2} implies, for each $T\in \triangle$,
\begin{equation*}
\begin{split}
\|\chi\|_{p,T}-K_1|T|^{2}|\chi|_{p,2,T}&\leq\underbrace{\left[3\sum_{i+j+k=1}^mw_{ijk}|\chi(\xi_{ijk})|^p\right]^{1/p}}_{\texttt{Mid}_T}\leq K_0\left(\|\chi\|_{p,T}+K_1|\triangle|^{2}|\chi|_{p,2,T}\right),
\end{split}\end{equation*}
where $K_0>0$ is some generic constant and $K_1>0$ is a constant depending only on $p$, since $d=1$.

From the mesh-dependent relation \eqref{equ:meshdependent} we have 
$$\|f\|_{p,\Omega}\leq \sum_{T\in\triangle}\|f\|_{p,T}.$$
If we regard $ \sum_{T\in\triangle}\|f\|_{p,T}$ as the $\ell_1$-vector norm of length $\#\triangle$, then
$$\sum_{T\in\triangle}\|f\|_{p,T}\leq |\#\triangle|^{1-1/p}\left[\sum_{T\in\triangle} \|f\|_{p,T}^p\right]^{1/p} = |\#\triangle|^{1-1/p}\|f\|_{p,\Omega}.$$
Similarly, this kind of estimates holds for the Sobolev semi-norm $|\cdot|_{p,2,T}$. Therefore,
\begin{equation}\label{equ:Mid1}
\begin{split}
\sum_{T\in\triangle}\texttt{Mid}_T 
&\leq K_0\left[\sum_{T\in\triangle}\|\chi\|_{p,T} + K_1|\triangle|^{2}\sum_{T\in\triangle}|\chi|_{p,2,T} \right]\\
&\leq K_0 |\#\triangle|^{1-1/p}(\|\chi\|_{p,\Omega}+K_1|\triangle|^{2}|\chi|_{p,2,\Omega})\\
& \leq K_0 |\#\triangle|^{1-1/p}\|\chi\|_{p,\Omega}+K_0K_1C_7 |\triangle|^{2/p}|\chi|_{p,2,\Omega},
\end{split}
\end{equation}
where the last equality is due to Lemma \ref{lem:number}. Conversely,
\begin{equation}\label{equ:Mid2}
\begin{split}
\sum_{T\in\triangle}\texttt{Mid}_T &\geq \sum_{T\in\triangle}\|\chi\|_{p,T} - K_1|\triangle|^{2}\sum_{T\in\triangle}|\chi|_{p,2,T} \\
&\geq \|\chi\|_{p,\Omega} - K_1 |\#\triangle|^{1-1/p}|\triangle|^{2}|\chi|_{p,2,\Omega}\\
&\geq \|\chi\|_{p,\Omega} - K_1 C_7|\triangle|^{2/p}|\chi|_{p,2,\Omega}.
\end{split}
\end{equation}

Note that
$$3\sum_{j=1}^mw_j|\chi(x_j,y_j)|^p = \sum_{T\in\triangle}\texttt{Mid}_T^p.$$
We play the same vector-norm trick for $\sum_{T\in\triangle}\texttt{Mid}_T$. For $d=1$, $w_{ijk}=A_T/3$. Therefore,
\begin{equation}\label{equ:Mid3}
\left[\sum_{T\in\triangle}\texttt{Mid}_T^p\right]^{1/p}\leq\sum_{T\in\triangle}\texttt{Mid}_T \leq |\#\triangle|^{1-1/p}\left[\sum_{T\in\triangle}\texttt{Mid}_T^p\right]^{1/p}.
\end{equation}

Together with \eqref{equ:Mid1} and \eqref{equ:Mid3}, we have
$$\left[\sum_{j=1}^mw_j|\chi(x_j,y_j)|^p\right]^{1/p}
\leq\frac{1}{3^{1/p}}\left(K_0 |\#\triangle|^{1-1/p}\|\chi\|_{p,\Omega}+K_0K_1C_7 |\triangle|^{2/p}|\chi|_{p,2,\Omega}\right),$$
and by Lemma \ref{lem:Markovpolygons} we have
\begin{equation}\label{equ:inter1}
\left[\sum_{j=1}^mw_j|\chi(x_j,y_j)|^p\right]^{1/p}\leq
\frac{1}{3^{1/p}}\left(K_0 |\#\triangle|^{1-1/p}+K_0K_1C_5C_7N^2 |\triangle|^{2/p}\right)\|\chi\|_{p,\Omega}.
\end{equation}
Together with \eqref{equ:Mid2} and \eqref{equ:Mid3}, we have
$$\left[\sum_{j=1}^mw_j|\chi(x_j,y_j)|^p\right]^{1/p}\geq \frac{1}{3^{1/p}|\#\triangle|^{1-1/p}}\left(\|\chi\|_{p,\Omega} - K_1 C_7|\triangle|^{2/p}|\chi|_{p,2,\Omega}\right),$$
and by Lemma \ref{lem:Markovpolygons}, we have 
\begin{equation}\label{equ:inter2}
\left[\sum_{j=1}^mw_j|\chi(x_j,y_j)|^p\right]^{1/p}\geq \frac{1 - K_1C_5C_7N^2|\triangle|^{2/p}}{3^{1/p}|\#\triangle|^{1-1/p}} \|\chi\|_{p,\Omega},
\end{equation}
where $N^2|\triangle|^{2/p}<1/c_1$ with $c_1:=K_1C_5C_7$ guarantees the positivity of the lower bound.

Combining with \eqref{equ:inter1} and \eqref{equ:inter2}, we have the inequality \eqref{equ:MZLp3} provided that 
\begin{equation*}\max\left\{\frac{\left(K_0 |\#\triangle|^{1-1/p}+K_0c_1N^2 |\triangle|^{2/p}\right)^p}{3}-1,
1-\frac{(1 - c_1N^2|\triangle|^{2/p})^p}{3|\#\triangle|^{p-1}}
\right\}\leq \eta. 
\end{equation*}

For $p=\infty$, the estimate \eqref{equ:MZIinf2} in Theorem \ref{thm:MZ2} implies, for each $T\in\triangle$,
\begin{equation*}
\|\chi\|_{\infty,T}-K_1|T|^{2}|\chi|_{\infty,2,T}\leq\max_{i+j+k=1}|\chi(\xi_{ijk})|\leq \|\chi\|_{\infty,T}+K_1|T|^{2}|\chi|_{\infty,2,T}.
\end{equation*}
Let $T_1$ be a triangle where $|\chi(x_j,y_j)|$ attains its maximum. Then it is immediate to estimate the upper bound in the form of 
$$\max_{1\leq j \leq m}|\chi(x_j,y_j)|\leq \|\chi\|_{\infty,T_1}+K_1|T_1|^{2}|\chi|_{\infty,2,T_1}\leq \|\chi\|_{\infty,\Omega}+K_1|\triangle|^{2}|\chi|_{\infty,2,T_1}.$$
For the lower bound, let $T_2$ be the triangle where $|\chi(x,y)|$ attains its maximum such that $\|\chi\|_{\infty,\Omega} =\|\chi\|_{\infty,T_2} = |\chi(x^*,y^*)|$ with $(x^*,y^*)\in T_2,$
then
$$\|\chi\|_{\infty,\Omega} - K_1|T_2|^2|\chi|_{\infty,2,T_2}\leq |\chi(x^*,y^*)|.$$ 
Let $(x_J,y_J)$ be a vertex of $T_2$, and we expand $\chi(x^*,y^*)$ at $(x_J,y_J)$:
$$\chi(x^*,y^*) = \chi(x_J,y_J) + \sum_{\alpha+\beta=1}D^{\alpha+\beta}\chi(v_j)|x^*-x_J|^{\alpha}|y^*-y_J|^{\beta},$$
where $v_j$ is some point between $(x^*,y^*)$ and $(x^J,y^J)$. Then 
$$|\chi(x^*,y^*) - \chi(x_J,y_J)|\leq 2|T_2||\chi|_{\infty,1,T_2}$$
and
$$|\chi(x^*,y^*)| - 2|T_2||\chi|_{\infty,1,T_2}\leq |\chi(x_J,y_J)|\leq \max_{1\leq j \leq m}|\chi(x_j,y_j)|.$$
Thus we have
\begin{equation*}\begin{split}
\max_{1\leq j \leq m}|\chi(x_j,y_j)|&\geq \|\chi\|_{\infty,\Omega} - K_1|T_2|^2|\chi|_{\infty,2,T_2} -2|T_2||\chi|_{\infty,1,T_2}\\
&\geq \|\chi\|_{\infty,\Omega} - C_4|\triangle||\chi|_{\infty,2,T_2}
\end{split}\end{equation*}

for some generic constant $C_4>0$. Therefore, we have
$$\|\chi\|_{\infty,\Omega} - C_4|\triangle||\chi|_{\infty,2,T_2}\leq\max_{1\leq j \leq m}|\chi(x_j,y_j)|\leq \|\chi\|_{\infty,\Omega}+K_1|\triangle|^{2}|\chi|_{\infty,2,T_1},$$
and by Lemma \ref{lem:Markovpolygons}, both Sobolev semi-norms can be bounded above by $C_5\|\chi\|_{\infty,\Omega}$. Therefore we have the inequality
$$(1 - C_4C_5N^2|\triangle|)\|\chi\|_{\infty,\Omega}\leq\max_{1\leq j \leq m}|\chi(x_j,y_j)|\leq (1+K_1C_5N^2|\triangle|^{2})\|\chi\|_{\infty,\Omega}$$
and hence the inequality \eqref{equ:MZpolygoninf}.
\end{proof}


\begin{corollary}[MZ for scattered data on triangles]
Let $T$ be a triangle in $\mathbb{R}^2$ and $\{(x_j,y_j)\}_{j=1}^m$ a set of scattered points in $T$ with vertices of $T$ included. Let $\triangle$ be a triangulation of $T$ generated from $\{(x_j,y_j)\}_{j=1}^m$. The quadrature rule \eqref{equ:quadP} in quadrature points $\{(x_j,y_j)\}_{j=1}^m$ with quadrature weights \eqref{equ:weightsP} satisfies the Marcinkiewicz--Zygmund inequalities established in Theorem \ref{thm:MZpolygon}.
\end{corollary}

\section{Numerical results}\label{sec:numerical}

For integrals over triangles and polygons, we test the following integrands
\begin{equation*}
\begin{split}
&f_1(x,y)=  (x+2y-7)^2 + (2x+y-5)^2,\\
&f_2(x,y) = 100\sqrt{|y-0.01x^2|} + 0.01|x+10|,\\
&f_3(x,y) =\sin(x+y) + (x-y)^2 -1.5x+2.5y+1
\end{split}
\end{equation*}
on a right triangle $\{(0,0),(0,1),(1,0)\}$ and a polygonal approximation of the map of Georgia, USA. We report the relative error defined as 
$$\texttt{relative error} = 
\frac{\left|\texttt{exact integral} - \texttt{quadrature evaluation}\right|}{\left|\texttt{exact integral}\right|}.$$
The exact integral is evaluated using a high-order, fine-mesh spline method.

For integration on the triangle, we use domain points \eqref{equ:domainpts} as quadrature points with corresponding weights \eqref{equ:weightsT}. The errors of quadrature rule \eqref{equ:quadT} with varying $d$ are reported in Table \ref{tab:tab1}. These errors decay as $d$ grows. Moreover, the errors with respect to $f_1(x,y)$, which is a polynomial of total degree 2, confirm the quadrature exactness of the rule \eqref{equ:quadT}.

\begin{table}[htbp]
\centering
\begin{tabular}{c|ccc}
\hline
$d$ & $f_1(x,y)$ & $f_2(x,y)$ &  $f_3(x,y)$ \\\hline
 1 & \texttt{1.9108e-02}& \texttt{3.0819e-01} & \texttt{2.1816e-01} \\\hline
 3 & \texttt{5.1213e-13}& \texttt{3.4226e-02} & \texttt{3.9214e-05} \\\hline
 5 & \texttt{5.1213e-13}& \texttt{9.1215e-03} & \texttt{7.2652e-08} \\\hline
 7 & \texttt{5.1132e-13}& \texttt{2.6919e-03} & \texttt{9.6764e-11} \\\hline
 9 & \texttt{5.0304e-13}& \texttt{3.6071e-04} & \texttt{4.1656e-13} \\\hline
11 & \texttt{4.5172e-13}& \texttt{6.3396e-04} & \texttt{4.7507e-13}\\\hline
\end{tabular} 
\caption{Relative error for the quadrature \eqref{equ:quadT} on a triangle.}
\label{tab:tab1}
\end{table}

For integration on the Georgia-like polygon, we select a sequence of points along the boundary, marked as red bullets. If no interior points are included, the weights \eqref{equ:weightsP} at these boundary points are labeled in the first plot of Figure \ref{fig:GA}. We can also incorporate some interior points, such as those representing Atlanta, Athens, and Macon. These three points are depicted as blue squares, and the new weights for all points are labeled in the second plot of Figure \ref{fig:GA}. The triangulation using all points is displayed in the third plot.

\begin{figure}[htbp]
\centering
\includegraphics[width = 0.3\textwidth]{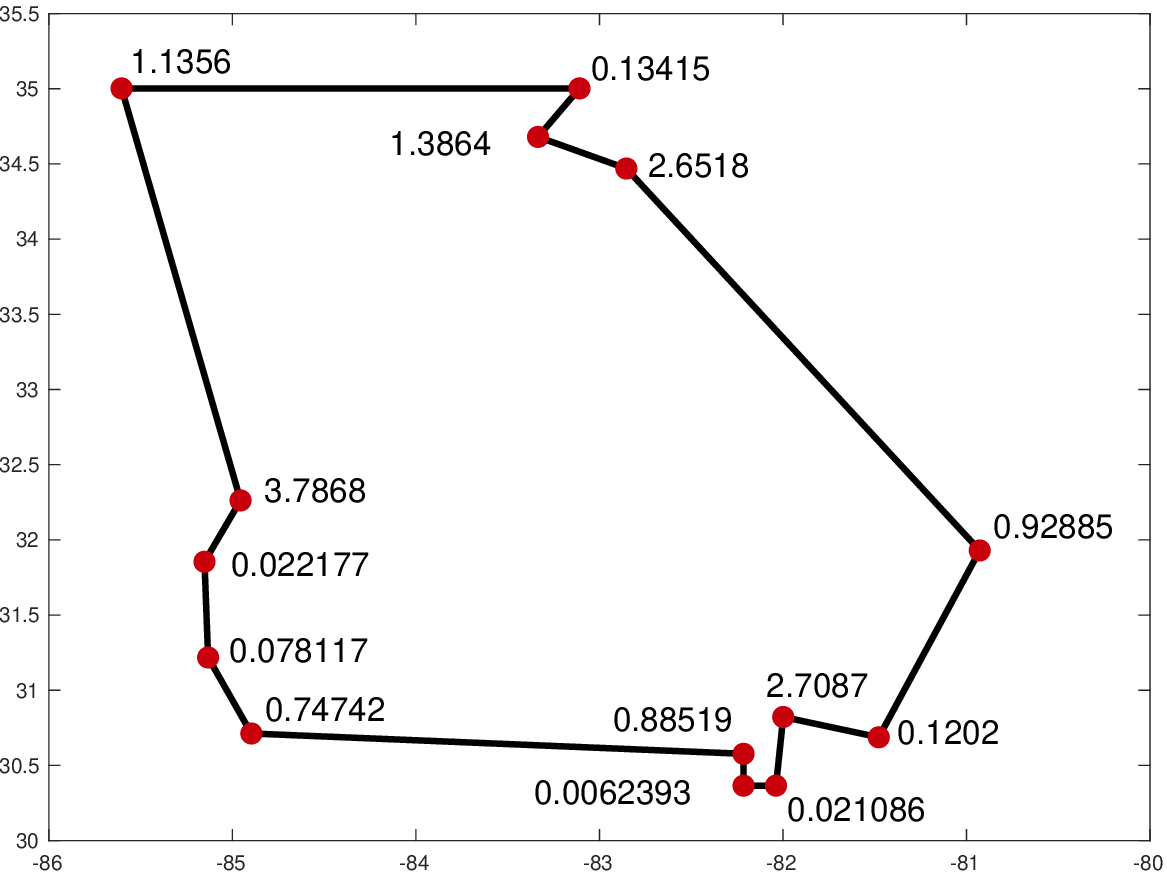}\quad\includegraphics[width = 0.3\textwidth]{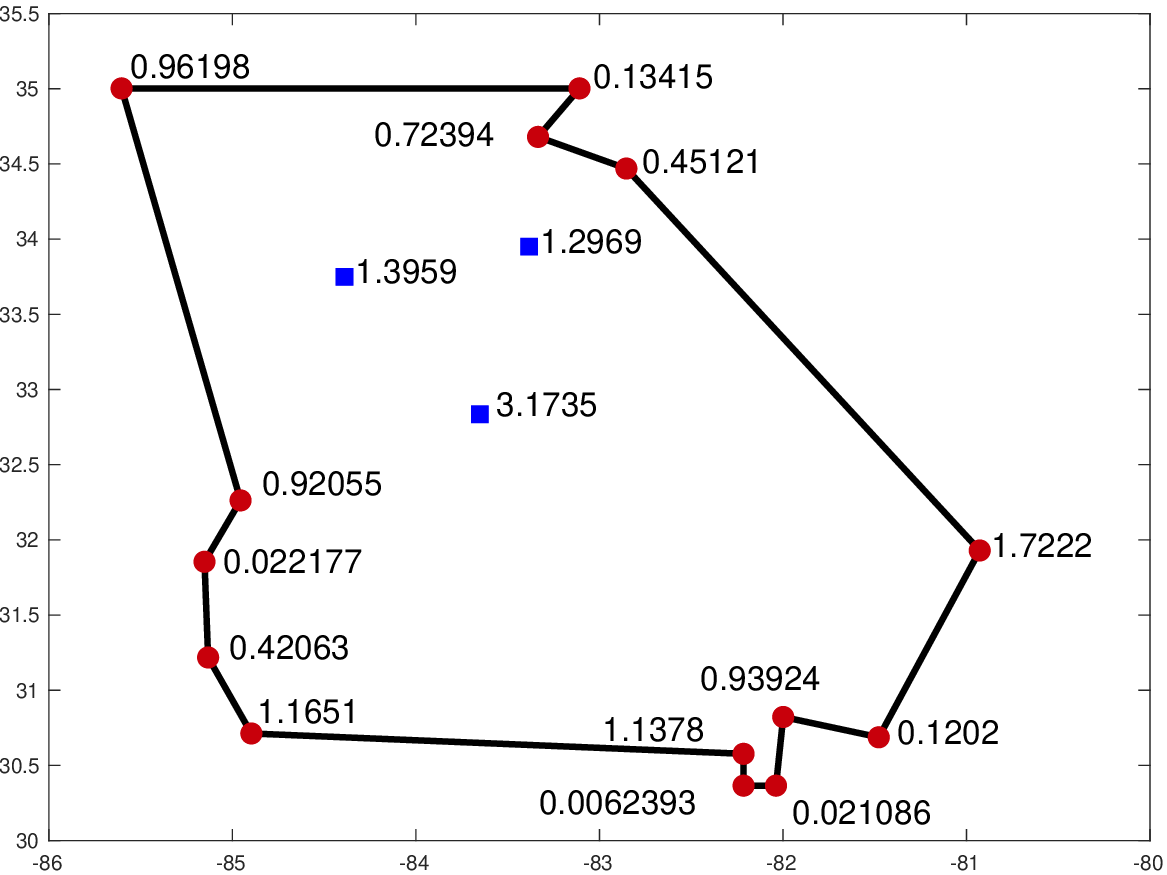}
\includegraphics[width = 0.3\textwidth]{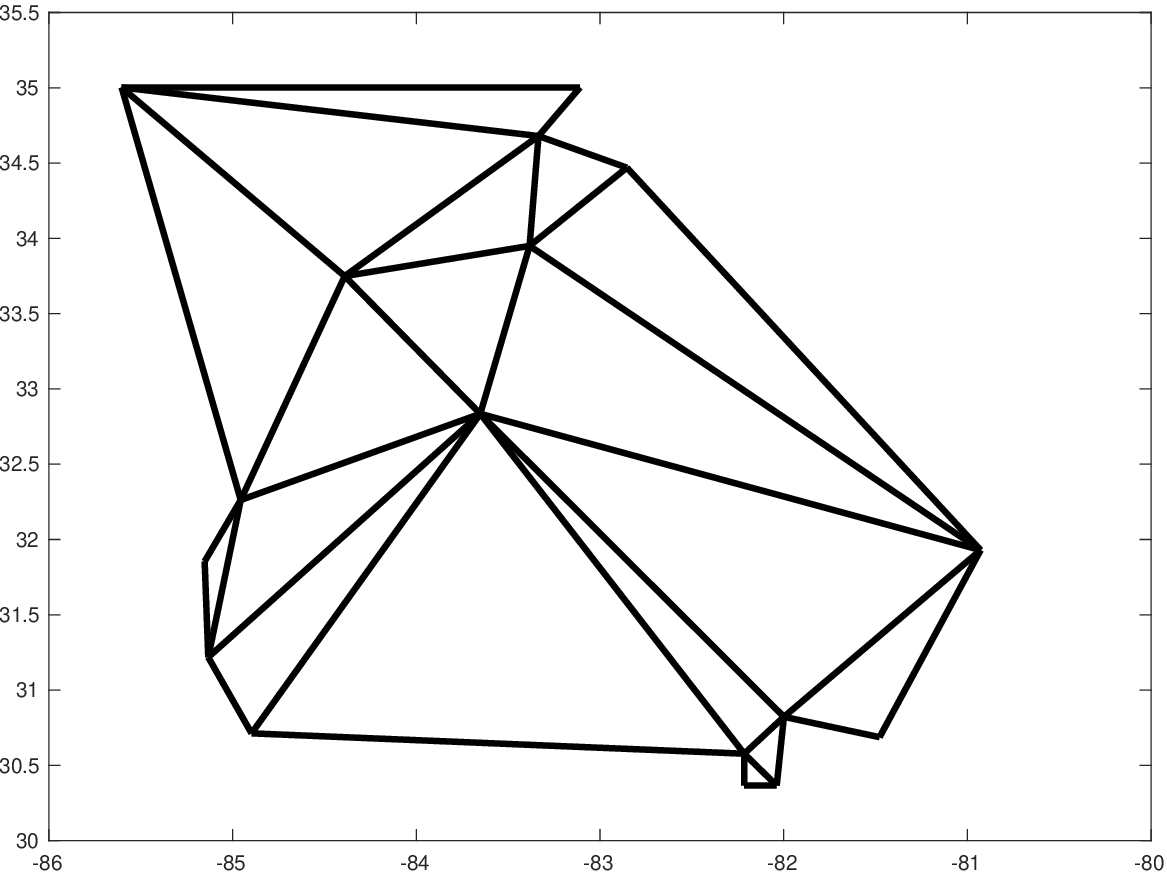}\\
\caption{Locations of quadrature points, corresponding quadrature weights, and triangulation.}\label{fig:GA}
\end{figure}

Errors of the quadrature \eqref{equ:quadP} for the testing integrands are listed in Table \ref{tab:tab2}, indicating that adding more points may reduce the quadrature error, provided the geometry of the triangulation is not compromised.

\begin{table}[htbp]
\centering
\begin{tabular}{c|ccc}
\hline
 & $f_1(x,y)$ & $f_2(x,y)$ &  $f_3(x,y)$ \\\hline
 boundary points & \texttt{4.4366e-04} & \texttt{1.5656e-04}    & \texttt{1.8794e-04} \\\hline
  boundary \& interior points & \texttt{2.2106e-04} &  \texttt{6.3152e-05}   &  \texttt{1.2770e-04}\\\hline 
\end{tabular} 
\caption{Relative error for the quadrature \eqref{equ:quadP} on a polygon.}
\label{tab:tab2}
\end{table}

\section{Final remark}

Given a set of scattered points on a polygon $\Omega\subset\mathbb{R}^2$, we have constructed a positive-weight quadrature rule that establishes the Marcinkiewicz--Zygmund inequalities $1\leq p\leq \infty$. This construction is inspired by our investigation on triangles, where we propose a quadrature rule and examine the related Marcinkiewicz--Zygmund estimate for $d = 1, 3, 5$. Although $d = 1$ suffices for the investigation on polygons, the quadrature rule for triangles offers additional interesting insights. For instance, when $d = 2$, as illustrated in Figure \ref{fig:location}, the weights at the vertices are zero. This corresponds to a toy Gaussian quadrature on triangles: a 3-point quadrature rule that is exact for polynomials of degree at most 2 (see, e.g., \cite{zbMATH03947511}). While the weights of our quadrature rule \eqref{equ:quadT} on triangles may be negative (hence our focus on $d = 1, 3, 5$), out arguments may apply to other positive-weight quadrature rules on triangles for studying Marcinkiewicz--Zygmund inequalities.

\section*{Acknowledgement}

I am very grateful to Professor Ming-Jun Lai for introducing me to splines and for his spline software.

\bibliographystyle{siamplain}
\bibliography{myref}
\clearpage

\end{document}